\documentclass[a4paper,reqno,10pt]{amsart}

\usepackage{a4wide}

\usepackage{amssymb}
\usepackage{amstext}
\usepackage{amsmath}
\usepackage{amscd}
\usepackage{amsthm}
\usepackage{amsfonts}

\usepackage{enumitem}
\setenumerate[1]{label={\upshape(\arabic*)}}
\setenumerate[2]{label={\upshape(\alph*)}}

\usepackage{tikz}
\usetikzlibrary{cd,arrows,decorations.pathmorphing,decorations.pathreplacing}

\newtheorem{theorem}{Theorem}[section]
\newtheorem{theoremi}{Theorem}

\newtheorem{corollary}[theorem]{Corollary}
\newtheorem{lemma}[theorem]{Lemma}
\newtheorem{proposition}[theorem]{Proposition}
\newtheorem{definition-proposition}[theorem]{Definition-Proposition}

\newtheorem{question}[theorem]{Question}

\theoremstyle{definition}
\newtheorem{definition}[theorem]{Definition}

\newtheorem{remark}[theorem]{Remark}
\newtheorem{example}[theorem]{Example}

\renewcommand{\AA}{\mathcal{A}}

\newcommand{\CC}{\mathcal{C}}

\newcommand{\DD}{\mathcal{D}}

\newcommand{\KKK}{\mathsf{K}}

\newcommand{\PP}{\mathcal{P}}

\newcommand{\JJ}{\mathcal{J}}

\newcommand{\Z}{\mathbb{Z}}

\newcommand{\Q}{\mathbb{Q}}

\newcommand{\length}{\operatorname{length}\nolimits}

\newcommand{\rad}{\operatorname{rad}\nolimits}

\newcommand{\Ext}{\operatorname{Ext}\nolimits}

\newcommand{\gl}{\operatorname{gl.\!dim}\nolimits}
\newcommand{\height}{\operatorname{ht}\nolimits}

\newcommand{\op}{\operatorname{op}\nolimits}
\newcommand{\RHom}{\mathbf{R}\strut\kern-.2em\operatorname{Hom}\nolimits}

\newcommand{\Kernel}{\operatorname{Ker}\nolimits}
\newcommand{\Cokernel}{\operatorname{Coker}\nolimits}
\newcommand{\Spec}{\operatorname{Spec}\nolimits}

\newcommand{\supp}{\operatorname{supp}\nolimits}

\newcommand{\Ab}{\mathcal{A}b}
\newcommand{\coker}{\Cokernel}

\renewcommand{\ker}{\Kernel}
\newcommand{\un}{\underline}

\DeclareMathOperator{\moduleCategory}{\mathsf{mod}} \renewcommand{\mod}{\moduleCategory}
\DeclareMathOperator{\Mod}{\mathsf{Mod}}

\DeclareMathOperator{\ind}{\mathsf{ind}}

\DeclareMathOperator{\Ex}{\mathsf{Ex}}
\DeclareMathOperator{\AR}{\mathsf{AR}}

\DeclareMathOperator{\CM}{\mathsf{CM}}
\DeclareMathOperator{\GP}{\mathsf{GP}}

\DeclareMathOperator{\add}{\mathsf{add}}

\DeclareMathOperator{\id}{\mathsf{id}}

\DeclareMathOperator{\eff}{\mathsf{eff}}

\newcommand{\iso}{\cong}
\newcommand{\infl}{\rightarrowtail}
\newcommand{\defl}{\twoheadrightarrow}

\newcommand{\equi}{\simeq}

\newenvironment{sbmatrix}{\left[\begin{smallmatrix}}{\end{smallmatrix}\right]}

\renewcommand{\AA}{\mathcal{A}}

\newcommand{\EE}{\mathcal{E}}

\numberwithin{equation}{section}

\begin{document}
\title[Relations for Grothendieck groups and representation-finiteness]{Relations for Grothendieck groups and representation-finiteness}

\author[H. Enomoto]{Haruhisa Enomoto}

\address{Graduate School of Mathematics, Nagoya University, Chikusa-ku, Nagoya. 464-8602, Japan}
\email{m16009t@math.nagoya-u.ac.jp}
\begin{abstract}
 For an exact category $\mathcal{E}$, we study the Butler's condition ``AR=Ex": the relation of the Grothendieck group of $\mathcal{E}$ is generated by Auslander-Reiten conflations. Under some assumptions, we show that AR=Ex is equivalent to that $\mathcal{E}$ has finitely many indecomposables. This can be applied to functorially finite torsion(free) classes and contravariantly finite resolving subcategories of the module category of an artin algebra, and the category of Cohen-Macaulay modules over an order which is Gorenstein or has finite global dimension. Also we showed that under some weaker assumption, AR=Ex implies that the category of syzygies in $\mathcal{E}$ has finitely many indecomposables.
\end{abstract}

\maketitle

\tableofcontents

\section{Introduction}
Let $\Lambda$ be a finite-dimensional $k$-algebra over a field $k$. To the abelian category $\mod\Lambda$ of finitely generated $\Lambda$-modules, we can associate an abelian group $\KKK_0(\mod\Lambda)$ called the \emph{Grothendieck group}. This is the quotient group $\KKK_0(\mod\Lambda,0) / \Ex (\mod\Lambda)$, where $\KKK_0(\mod\Lambda,0)$ denotes the free abelian group with the basis the set of isomorphism classes $[X]$ of indecomposable objects $X \in \mod\Lambda$, and $\Ex(\mod\Lambda)$  the subgroup generated by $[X]-[Y]+[Z]$ for every exact sequence $0 \to X \to Y \to Z\to 0$.

\emph{Auslander-Reiten sequences}, or \emph{AR sequences} for short, are special kind of short exact sequences in $\mod\Lambda$, which are ``minimal'' in some sense. We denote by $\AR(\mod\Lambda)$ the subgroup of $\Ex(\mod\Lambda)$ generated by AR sequences. Then Butler and Auslander proved the following result.
\begin{theorem}[\cite{but,relations}]
Let $\Lambda$ be a finite-dimensional algebra. Then $\AR(\mod\Lambda) = \Ex(\mod\Lambda)$ holds if and only if there are only finitely many indecomposable objects in $\mod\Lambda$ up to isomorphism.
\end{theorem}

Auslander conjectured that similar results hold for a more general class of categories other than $\mod\Lambda$. In this paper, we study this conjecture in the context of Quillen's \emph{exact categories}.
For a Krull-Schmidt exact category $\EE$, the Grothendieck group $\KKK_0(\EE)= \KKK_0(\EE,0)/\Ex(\EE)$ of $\EE$ is defined in the same way. Also the notion of AR sequences is defined over $\EE$, which we call \emph{AR conflations}. Thus we have the subgroup $\AR(\EE)$ of $\Ex(\EE)$, and the aim of this paper is to study the following question.
\begin{question}\label{q}
For a Krull-Schmidt exact category $\EE$, when are the following equivalent?
\begin{enumerate}
\item $\EE$ is of finite type, that is, there exist only finitely many indecomposable objects in $\EE$ up to isomorphism.
\item $\AR(\EE) = \Ex(\EE)$ holds.
\end{enumerate}
\end{question}

Let us mention some known results on this question. For the implication (1) $\Rightarrow$ (2), we have the following quite general result by the author:
\begin{theorem}[{\cite[Theorem 3.18]{en2}, Corollary \ref{finadm}}]
  Let $R$ be a complete noetherian local ring and $\EE$ a Krull-Schmidt exact $R$-category such that $\EE(X,Y)$ is a finitely generated $R$-module for each $X,Y \in \EE$. If $\EE$ has enough projectives and is of finite type, then $\AR(\EE) = \Ex(\EE)$ holds.
\end{theorem}
This can be applied to various situations in the representation theory of noetherian algebras. We provide a proof of this in Corollary \ref{finadm}.

On the other hand, the implication (2) $\Rightarrow$ (1) is more subtle. Indeed, some counter-examples are known, e.g. \cite[Section 3]{mmp}. However, there are affirmative results for this implication in some concrete situations: $\EE = \CM \Lambda$ for an $R$-order $\Lambda$ with finite global dimension \cite[Proposition 2.3]{ar1}, and the category of modules with good filtrations over standardly stratified algebras \cite{mmp,pr}. Also some partial results are known (\cite{hir,ko}).

In this paper, we first give a sufficient condition for (1) and (2) in Question \ref{q} to be equivalent, by using a functorial method. For a Krull-Schmidt exact category $\EE$, we consider the category $\Mod\EE$ of functors $\EE^{\op} \to \Ab$. There is an important subcategory $\eff\EE$ of $\Mod\EE$, the category of \emph{effaceable} functors (Definition \ref{effdef}), which nicely reflects the exact structure of $\EE$ and is closely related to the group $\Ex(\EE)$.
Then let us consider the following conditions:
\begin{enumerate}[label={\upshape(\alph*)}]
\item $\EE$ is of finite type.
\item $\EE$ is \emph{admissible} \cite{en2}, that is, every object in $\eff\EE$ has finite length in $\Mod\EE$.
\item $\AR(\EE)=\Ex(\EE)$ holds.
\end{enumerate}
Simple modules contained in $\eff\EE$ bijectively correspond to AR conflations in $\EE$ (Proposition \ref{exsimp}). Thus the condition (b) can be regarded as a functorial analogue of (c), hence it is somewhat closer to (c) than (a) is. The relation of these conditions are summarized in Figure \ref{fig}.

Secondly, we apply this condition to concrete situations and obtain results on Question \ref{q}.
For an artin algebra, we show the following result, which extends the results \cite[Theorem 3]{mmp} and \cite[Theorem 3.6]{pr}.
\begin{theoremi}[{\upshape = Theorem \ref{artinmain}}]\label{cora}
Let $\Lambda$ be an artin algebra and $\EE$ a contravariantly finite resolving subcategory of $\mod\Lambda$. Then the following are equivalent.
\begin{enumerate}
\item $\EE$ is of finite type.
\item $\AR(\EE) = \Ex(\EE)$ holds.
\end{enumerate}
\end{theoremi}
When $R$ is not artinian, we prove the following result. This extends \cite{hir,ko} to the non-commutative case, and \cite[Proposition 2.3]{ar1} to Gorenstein orders. Also our method provides another proof of \cite[Proposition 2.3]{ar1}.
\begin{theoremi}[{\upshape = Theorem \ref{ordermain}, Corollary \ref{syzygymain}}]
Let $R$ be a complete Cohen-Macaulay local ring and $\Lambda$ an $R$-order with at most an isolated singularity. Suppose that $\Ex(\CM \Lambda) = \AR(\CM \Lambda)$ holds. Then the following hold.
\begin{enumerate}
\item $\mathsf{\Omega CM}\,\Lambda$, the category of syzygies of Cohen-Macaulay $\Lambda$-modules, is of finite type.
\item If $\Lambda$ has finite global dimension or $\Lambda$ is a Gorenstein order, then $\CM\Lambda$ is of finite type.
\end{enumerate}
\end{theoremi}

Actually, most of the results above are deduced from the following relations between (a), (b) and (c):
\begin{theoremi}[Proposition \ref{propab}, Theorem \ref{fcfl}]
  Let $\EE$ be a Krull-Schmidt exact $R$-category with a projective generator over a commutative noetherian ring $R$. Then the implications in Figure \ref{fig} hold.
\end{theoremi}
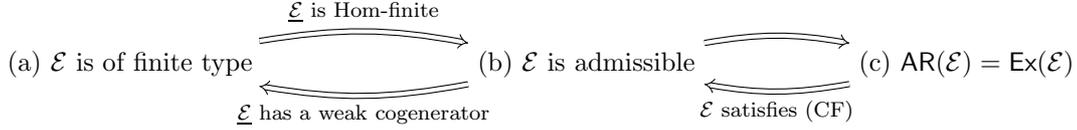
\begin{figure}[h]
\begin{tikzpicture}
 \node (a) at (-6,0)  {(a) $\EE$ is of finite type};
 \node (b) at (0,0)  {(b) $\EE$ is admissible};
 \node (c) at (5,0) {(c) $\AR(\EE) = \Ex(\EE)$};

 \draw [-implies,double equal sign distance] (a) to [bend left=10] node[auto] {\footnotesize $\un{\EE}$ is Hom-finite} (b) ;
 \draw [-implies,double equal sign distance] (b) to [bend left=10] node[auto] {\footnotesize $\un{\EE}$ has a weak cogenerator} (a) ;
 \draw [-implies,double equal sign distance] (b) to [bend left=10] (c) ;
 \draw [-implies,double equal sign distance] (c) to [bend left=10] node[auto] {\footnotesize $\EE$ satisfies (CF)} (b) ;
\end{tikzpicture}
\caption{Summary of Section \ref{sec3}}\label{fig}
\end{figure}
Here we introduce the condition \emph{(CF)} on exact categories concerning the Grothendieck groups (see Definition \ref{fcdef}).

The paper is organized as follows.
In Section \ref{sec2}, we introduce basic definitions and collect some results about Grothendieck group and effaceable functors.
In Section \ref{sec3}, we introduce some conditions on $\EE$ and prove the relations indicated in Figure \ref{fig}.
In Section \ref{sec4}, we verify that the conditions are satisfied in some concrete cases.
In Section \ref{sec5}, we study the finiteness of the category of syzygies.
In Appendix A, we provide some interesting properties of the category of effaceable functors.

\subsection{Notation and conventions}
Throughout this paper, we assume that \emph{all categories and functors are additive} and \emph{all subcategories are additive, full and closed under isomorphisms}. For an object $G$ of an additive category $\EE$, we denote by $\add G$ the category consisting of all direct summands of finite direct sums of $G$. We say that an additive category is \emph{of finite type} if there exists an object $G \in \EE$ such that $\EE = \add G$ holds.

For a Krull-Schmidt category $\EE$, we denote by $\ind\EE$ the set of isomorphism classes of indecomposable objects in $\EE$. We refer the reader to \cite{krause} for the basics of Krull-Schmidt categories.
We say that an additive category $\EE$ is an \emph{idempotent complete} if every idempotent morphism splits. It is well-known that a Krull-Schmidt category is idempotent complete.

As for exact categories, we use the terminologies \emph{inflations}, \emph{deflations} and \emph{conflations}. We refer the reader to \cite{buhler} for the basics of exact categories.
We say that an exact category $\EE$ has a \emph{projective generator $P$} if $\EE$ has enough projectives $\add P$.

\section{Preliminaries}\label{sec2}
First we introduce the basic definition about Auslander-Reiten theory in a Krull-Schmidt category. Let $\EE$ be a Krull-Schmidt category and $\JJ$ its Jacobson radical. A morphism $g:Y\to Z$ in $\EE$ is called \emph{right almost split} if $Z$ is indecomposable, $g$ is in $\JJ$ and any morphism $h:W \to Z$ in $\JJ$ factors through $g$. Dually we define \emph{left almost split}.

For a Krull-Schmidt exact category $\EE$, we say that a conflation $0 \to X \xrightarrow{f} Y \xrightarrow{g} Z \to 0$ in $\EE$ is an \emph{AR conflation} if $f$ is left almost split and $g$ is right almost split. We say that an indecomposable object $Z$ in $\EE$ \emph{admits a right AR conflation} if there exists an AR conflation ending at $Z$. We say that $\EE$ \emph{has right AR conflations} if every indecomposable non-projective object admits a right AR conflation. Left AR conflations are defined dually, and we say that $\EE$ \emph{has AR conflations} if it has both right and left AR conflations.

Next we introduce some notation concerning the Grothendieck group. For a Krull-Schmidt exact category $\EE$, let $\KKK_0(\EE,0)$ be a free abelian group $\bigoplus_{[X] \in \ind \EE} \mathbb{Z}\cdot[X]$ generated by the set $\ind \EE$ of isomorphism classes of indecomposable objects in $\EE$, which we identify with the Grothendieck group of $\EE$ with the split exact structure, that is, the exact structure where all conflations are split exact.
We denote by $\Ex(\EE)$ the subgroup of $\KKK_0(\EE,0)$ generated by
\[
\{ [X] - [Y] + [Z] \text{ $|$ there exists a conflation $X \infl Y \defl Z$ in $\EE$} \}.
\]
We call the quotient group $\KKK_0(\EE) :=\KKK_0(\EE,0) / \Ex(\EE)$ the \emph{Grothendieck group} of $\EE$.
We denote by $\AR(\EE)$ the subgroup of $\Ex(\EE)$ generated by
\[
\{ [X] - [Y] + [Z] \text{ $|$ there exists an AR conflation $X \infl Y \defl Z$ in $\EE$} \}.
\]

Throughout this paper, we make use of functor categorical arguments. Let us recall the related concepts. For an additive category $\EE$, a \emph{right $\EE$-module} $M$ is a contravariant additive functor $M: \EE^{\op} \to \Ab$ from $\EE$ to the category of abelian groups $\Ab$. We denote by $\Mod\EE$ the category of right $\EE$-modules, and morphisms are natural transformations between them.
Then the category $\Mod \EE$ is a Grothendieck abelian category with enough projectives, and projective objects are precisely direct summands of (possibly infinite) direct sums of representable functors.

We denote by $\mod \EE$ the category of finitely presented $\EE$-modules, that is, the modules $M$ such that there exists an exact sequence $\EE(-,X) \to \EE(-,Y) \to M \to 0$ for some $X,Y$ in $\EE$. This category $\mod\EE$ is not necessarily an abelian category, but $\mod\EE$ is closed under extensions in $\Mod\EE$ by the horseshoe lemma. Thereby we always regard $\mod\EE$ as an exact category, which has enough projectives in the sense of exact category. Moreover, if $\EE$ is idempotent complete, then projectives in $\mod\EE$ are precisely representable functors.

Now let us introduce the notion of \emph{effaceability} of $\EE$-modules, which plays an essential role throughout this paper. This notion was originally introduced by Grothendieck \cite[p. 141]{gro}.
\begin{definition}\label{effdef}
Let $\EE$ be an exact category. A right $\EE$-module $M$ is called \emph{effaceable} if there exists a deflation $f:Y \defl Z$ in $\EE$ such that $\EE(-,Y) \xrightarrow{\EE(-,f)} \EE(-,Z) \to M \to 0$ is exact. We denote by $\eff \EE$ the subcategory of $\mod\EE$ consisting of effaceable functors.
\end{definition}
The category $\eff\EE$ plays a quite essential role to study exact categories in a functorial method.
This definition for exact categories is due to \cite{fiorot}, and the category $\eff\EE$ plays an important role in \cite{en2} to classify exact structures on a given additive category. For the compatibility of Grothendieck's effaceability, see Proposition \ref{grodef}.

In the rest of this section, we collect basic properties about functor categories and effaceable functors which we need in the sequel.
\begin{proposition}\label{kssimp}
Let $\EE$ be a Krull-Schmidt category and $\JJ$ its Jacobson radical. Put $S_Z := \EE(-,Z)/\JJ(-,Z)$ for $Z \in \EE$. Then the following hold.
\begin{enumerate}
\item The map $Z \mapsto S_Z$ gives a bijection between $\ind\EE$ and the set of isomorphism classes of simple $\EE$-modules.
\item A morphism $g:Y \to Z$ in $\EE$ is right almost split if and only if $Z$ is indecomposable and $\EE(-,Y) \xrightarrow{\EE(-,g)} \EE(-,Z) \to S_Z \to 0$ is exact.
\end{enumerate}
\end{proposition}
\begin{proof}
It is well-known, see \cite{artin2} for example.
\end{proof}
In particular, simple $\EE$-modules contained in $\mod\EE$ bijectively correspond to indecomposable objects in $\EE$ which admit right almost split maps.
Moreover if $\EE$ is an exact category, then one can show more on $\eff\EE$ as follows.
\begin{proposition}\label{exsimp}
Let $\EE$ be a Krull-Schmidt exact category and $Z$ an indecomposable object in $\EE$. Then the following are equivalent.
\begin{enumerate}
\item $Z$ admits a right AR conflation.
\item $S_Z$ belongs to $\eff\EE$.
\item $S_Z$ is finitely presented and $Z$ is non-projective.
\end{enumerate}
In particular, simple $\EE$-modules contained in $\eff\EE$ bijectively correspond to indecomposables in $\EE$ which admit right AR conflations.
\end{proposition}
\begin{proof}
(1) $\Rightarrow$ (2):
We have a deflation $g:Y\defl Z$ which is right almost split. Thus $S_Z \iso \coker \EE(-,g)$ belongs to $\eff\EE$.

(2) $\Rightarrow$ (3):
Since $S_Z$ is in $\eff\EE$, it is finitely presented. By \cite[Propositions 2.11]{en2}, we have that $Z$ is non-projective since $S_Z(Z) \neq 0$ holds.

(3) $\Rightarrow$ (1):
This is precisely \cite[Proposition A.3]{en2}.
\end{proof}

The following says that the category $\eff\EE$ has nice properties and controls the exact structure of $\EE$. Here we say that an $\EE$-module $M$ is \emph{finitely generated} if there exists a surjection from a representable functor onto $M$.
\begin{proposition}\label{serre}
Let $\EE$ be an idempotent complete exact category. Then the following holds.
\begin{enumerate}
\item Suppose that there exists an exact sequence $0 \to M_1 \to M \to M_2 \to 0$ in $\Mod\EE$ and that $M_1$ is finitely generated. Then $M$ is in $\eff\EE$ if and only if both $M_1$ and $M_2$ are in $\eff\EE$.
\item A complex $X \xrightarrow{f} Y \xrightarrow{g} Z$ in $\EE$ is a conflation if and only if there exists an exact sequence
\[
0 \to \EE(-,X) \xrightarrow{\EE(-,f)} \EE(-,Y) \xrightarrow{\EE(-,g)} \EE(-,Z) \to M \to 0
\]
in $\Mod\EE$ with $M$ in $\eff\EE$.
\end{enumerate}
\end{proposition}
\begin{proof}
We refer the reader to \cite[Proposition 2.10]{en2} for the proof of (1), and \cite[Theorem 2.7]{en2} or the proof of Lemma \ref{grodef} for (2).
\end{proof}

The category $\eff\EE$ enjoys further nice properties, e.g. being abelian, and we refer the interested reader to Appendix A.

If $\EE$ has enough projectives, then one has a simpler description of the category $\eff\EE$. Let $\EE$ be an exact category with enough projectives and denote by $\PP$ the subcategory of $\EE$ consisting of projective objects. We define the \emph{stable category} $\un{\EE}$ by the ideal quotient $\EE/[\PP]$. Then $\eff\EE$ can be naturally identified with the module category of $\un{\EE}$, which we omit the proof.
\begin{lemma}[{\cite[Lemma 2.13]{en2}}]\label{effstable}
Let $\EE$ be an exact category with enough projectives $\PP$. Then the natural restriction functor $\Mod\un{\EE} \to \Mod\EE$ induces an equivalence $\mod\un{\EE} \equi \eff \EE$.
\end{lemma}

\section{General conditions for (a), (b) and (c) to be equivalent}\label{sec3}
In this section, we will provide some general results on the three conditions: (a) $\EE$ is of finite type, (b) $\EE$ is admissible and (c) $\AR(\EE) = \Ex(\EE)$. Here we say that $\EE$ is \emph{admissible} if every object in $\eff\EE$ has finite length in $\Mod\EE$ (\cite{en2}). See Proposition \ref{onaji} for the equivalent condition of the admissibility. Then our results in this section can be summarized in Figure \ref{fig} in the introduction.

To begin with, let us introduce some terminologies on $R$-categories needed in this paper. Let $R$ be a commutative noetherian ring and $\EE$ an $R$-category. We say that an $R$-category $\EE$ is \emph{Hom-noetherian} (resp. \emph{Hom-finite}) if $\EE(X,Y)$ is finitely generated (resp. has finite length) as an $R$-module for each $X,Y$ in $\EE$.

To deal with admissibility, it is convenient to introduce the following criterion for deciding whether the functor has finite length or not.
Let $\EE$ be a Krull-Schmidt category and $M \in \Mod\EE$ a right $\EE$-module. Then the support $\supp M$ of $M$ is defined by
\[
\supp M := \{ X \in \ind\EE \ |\  M(X) \neq 0\}.
\]
The Hom-finiteness ensures that finiteness of modules can be checked only by considering their support:
\begin{lemma}\label{fl}
Let $\EE$ be a Krull-Schmidt category and $M \in \mod\EE$ an $\EE$-module. If $M$ has finite length in $\Mod\EE$, then $\supp M$ is a finite set. The converse holds if $\EE$ is Hom-finite $R$-category.
\end{lemma}
\begin{proof}
See, for example, \cite[Theorem 2.12]{artin2}.
\end{proof}

\subsection{(a) Representation-finiteness and (b) admissibility}\label{secab}
First, we consider the relation between: (a) $\EE$ is of finite type, and (b) $\EE$ is admissible.

\begin{definition}
Let $\CC$ be an additive category. We say that an object $C \in \CC$ is a \emph{weak cogenerator} if $\CC(X,C)=0$ implies that $X$ is a zero object for every $X \in \CC$.
\end{definition}
If we assume the existence of a weak cogenerator of the stable category, then (a) is almost equivalent to (b).
\begin{proposition}\label{propab}
Let $\EE$ be a Krull-Schmidt exact category with a projective generator. Consider the following conditions:
\begin{enumerate}
\item $\EE$ is admissible and $\un{\EE}$ has a weak cogenerator.
\item $\EE$ is of finite type.
\end{enumerate}
Then {\upshape (1)} $\Rightarrow$ {\upshape (2)} holds. The converse holds if $\un{\EE}$ is a Hom-finite $R$-category for a commutative noetherian ring $R$.
\end{proposition}
\begin{proof}
(1) $\Rightarrow$ (2):
Since $\EE$ has enough projectives, we may identify $\eff\EE$ with $\mod\un{\EE}$ by Lemma \ref{effstable}. Let $\un{X}$ be a weak cogenerator of $\un{\EE}$. By the assumption, we have that $\un{\EE}(-,\un{X})$ has finite length. This implies that $\supp \un{\EE}(-,\un{X})$ is finite by Lemma \ref{fl}. Since every indecomposable object $\un{W}$ in $\un{\EE}$ should satisfy  $\un{\EE}(\un{W},\un{X})\neq 0$, it follows that $\un{\EE}$ is of finite type. Now $\EE$ itself is of finite type because we have a natural identification $\ind\EE =\ind\un{\EE} \sqcup \ind (\add P)$.

(2) $\Rightarrow$ (1):
Suppose that $\EE$ is of finite type and $\un{\EE}$ is Hom-finite over $R$. Then obviously the direct sum of all indecomposables in $\un{\EE}$ is a weak cogenerator of $\un{\EE}$. Moreover, since $\un{\EE}$ is Hom-finite and $\ind\un{\EE}$ is finite, every object in $\eff\EE = \mod\un{\EE}$ has finite length by Lemma \ref{fl}.
\end{proof}

\subsection{(b) Admissibility and (c) AR=Ex}\label{secbc}

Now let us investigate the relation between the condition (b) and (c).
First we show that if (b) $\EE$ is admissible, then (c) $\AR(\EE) =\Ex(\EE)$ holds. Next we will introduce the \emph{condition (CF)} and show that the converse implication holds under it.

\begin{proposition}\label{part1}
Let $\EE$ be a Krull-Schmidt exact category. Suppose that $\EE$ is admissible. Then $\Ex(\EE) = \AR(\EE)$ holds.
\end{proposition}
\begin{proof}
This is proved in the similar way as \cite[Theorem 3.17]{en2}. We provide a proof here for the convenience of the reader.

Let $0 \to X_1 \xrightarrow{f_1} Y_1 \xrightarrow{g_1} Z_1 \to 0$ be a conflation in $\EE$ and we have to show $[X_1] - [Y_1] + [Z_1] \in \AR(\EE)$.
Put $M:= \coker \EE(-,g_1)$ in $\Mod\EE$, then $M$ is in $\eff\EE$ by the definition.
Suppose that there exists another conflation $0 \to X_2 \xrightarrow{f_2} Y_2 \xrightarrow{g_2} Z_2 \to 0$ in $\EE$ such that $M \iso \coker \EE(-,g_2)$. Then we have exact sequences
\[
0 \to {\EE(-,X_i)} \xrightarrow{\EE(-,f_i)} \EE(-,Y_i) \xrightarrow{\EE(-,g_i)} \EE(-,Z_i) \to M \to 0
\]
in $\Mod\EE$ for $i=1,2$. Thus Schanuel's lemma shows that $\EE(-,X_1) \oplus \EE(-,Y_2) \oplus \EE(-,Z_1) \iso \EE(-,X_2) \oplus \EE(-,Y_1) \oplus \EE(-,Z_2)$ in $\Mod\EE$, hence $X_1 \oplus Y_2 \oplus Z_1 \iso X_2 \oplus Y_1 \oplus Z_2$
 in $\EE$ by the Yoneda lemma. This implies that $[X_1] - [Y_1] + [Z_1] = [X_2] - [Y_2] + [Z_2]$ in $\KKK_0(\EE,0)$. Thus it suffices to show the following claim.

\emph{Claim: For any $M \in \eff \EE$, there exists at least one exact sequence in $\Mod\EE$
\[
0 \to \EE(-,X) \xrightarrow{\EE(-,f)} \EE(-,Y) \xrightarrow{\EE(-,g)} \EE(-,Z) \to M \to 0
\] with $[X]-[Y]+[Z] \in \AR(\EE)$.} Note that such a complex $0 \to X \to Y \to Z \to 0$ is automatically a conflation by Proposition \ref{serre}(2).

We will show this claim by induction on $l(M)$, the length of $M$ as an $\EE$-module.
Suppose that $l(M) = 1$, that is, $M$ is a simple $\EE$-module. We may assume that $M=S_Z$ for an indecomposable object $Z \in \EE$ by Proposition \ref{kssimp}, and there exists an AR conflation $0 \to X \xrightarrow{f} Y \xrightarrow{g} Z \to 0$ in $\EE$ by Proposition \ref{exsimp}. This gives the desired projective resolution $0 \to \EE(-,X) \xrightarrow{\EE(-,f)} \EE(-,Y) \xrightarrow{\EE(-,g)} \EE(-,Z) \to S_Z \to 0$ of $S_Z$, and $[X] - [Y] + [Z] \in \AR(\EE)$ holds.

Now suppose that $l(M) > 1$. Take a simple $\EE$-submodule $M_1$ of $M$. Then we have an exact sequence $0 \to M_1 \to M \to M_2 \to 0$ in $\Mod\EE$. Since $M_1$ is finitely generated, Proposition \ref{serre}(1) shows that $M_1$ and $M_2$ are both in $\eff\EE$. By induction hypothesis, we have an exact sequence $0 \to \EE(-,X_i) \xrightarrow{\EE(-,f_i)} \EE(-,Y_i) \xrightarrow{\EE(-,g_i)} \EE(-,Z_i) \to M_i \to 0$ with $[X_i] -[Y_i] + [Z_i] \in \AR(\EE)$ for each $i=1,2$.
By the horseshoe lemma, we obtain a projective resolution
\[
0 \to \EE(-,X_1 \oplus X_2) \xrightarrow{\EE(-,f)} \EE(-,Y_1\oplus Y_2) \xrightarrow{\EE(-,g)} \EE(-,Z_1 \oplus Z_2) \to M \to 0.
\]
Then we have
\[
[X_1\oplus X_2 ] -[Y_1\oplus Y_2] + [Z_1 \oplus Z_2]
=\sum_{i=1,2}([X_i] -[Y_i] + [Z_i]) \in \AR(\EE).
\]
Thus the claim follows.
\end{proof}

Next let us investigate when the converse of Proposition \ref{part1} holds. The author does not know whether this always holds, but we show that this is the case for particular cases. To this purpose, let us introduce the condition \emph{(CF)}, which stands for the \emph{Conservation of Finiteness}.

\begin{definition}\label{fcdef}
For an exact category $\EE$, we say that $\EE$ satisfies \emph{(CF)} if it satisfies the following:
\begin{itemize}
\item[(CF)] Let $0 \to X \xrightarrow{f_i} Y \xrightarrow{g_i} Z \to 0$ be two conflations in $\EE$ for $i=1,2$ and put $M_i := \coker \EE(-,g_i) \in \eff\EE$. If $M_1$ has finite length in $\Mod\EE$, then so does $M_2$.
\end{itemize}
\end{definition}
For later purposes, the following equivalent condition is more useful.
\begin{lemma}
For an exact category $\EE$, the condition {\upshape (CF)} is equivalent to the following condition.
\begin{itemize}
\item[\upshape (CF$'$)] Let $0 \to X_i \xrightarrow{f_i} Y_i \xrightarrow{g_i} Z_i \to 0$ be conflations in $\EE$ for $i=1,2$ and put $M_i := \coker \EE(-,g_i)$ in $\eff\EE$.
Suppose that $[X_1]-[Y_1]+[Z_1]=[X_2]-[Y_2]+[Z_2]$ holds in $\KKK_0(\EE,0)$. If $M_1$ has finite length in $\Mod\EE$, then so does $M_2$.
\end{itemize}
\end{lemma}
\begin{proof}
Obviously it suffices to show that (CF) implies (CF$'$). Since $[X_1]-[Y_1]+[Z_1]=[X_2]-[Y_2]+[Z_2]$ holds in $\KKK_0(\EE,0)$, we have that $X_2 \oplus Y_1 \oplus Z_2 \iso X_1 \oplus Y_2 \oplus Z_1$ holds in $\EE$ by the Krull-Schmidtness of $\EE$. Now consider the following two complexes:
\begin{align*}
0 \to X_1 \oplus X_2 \xrightarrow{
\begin{sbmatrix} 0 & 1 \\ f_1 & 0 \\ 0 & 0 \end{sbmatrix}
} X_2 \oplus Y_1 \oplus Z_2 \xrightarrow{
\begin{sbmatrix}
0 & g_1 & 0 \\
0 & 0 & 1
\end{sbmatrix}
} Z_1 \oplus Z_2 \to 0,\\
0 \to X_1 \oplus X_2 \xrightarrow{
\begin{sbmatrix} 1 & 0 \\ 0 & f_2 \\ 0 & 0 \end{sbmatrix}
} X_1 \oplus Y_2 \oplus Z_1 \xrightarrow{
\begin{sbmatrix}
0 & 0 & 1 \\
0 & g_2 & 0
\end{sbmatrix}
} Z_1 \oplus Z_2 \to 0.
\end{align*}
These are conflations since they are direct sums of conflations and split exact sequences. Applying (CF) to them, it is easy to check that (CF$'$) holds.
\end{proof}

Now we can state the following main result in this section.
\begin{theorem}\label{fcfl}
Let $\EE$ be a Krull-Schmidt exact category. Then the following are equivalent.
\begin{enumerate}
\item $\EE$ is admissible.
\item $\Ex(\EE) = \AR(\EE)$ holds and $\EE$ satisfies {\upshape (CF)}.
\item $\Ex(\EE)\otimes_\Z \Q = \AR(\EE) \otimes_\Z \Q$ holds and $\EE$ satisfies {\upshape (CF)}.
\end{enumerate}
\end{theorem}
\begin{proof}
(1) $\Rightarrow$ (2):
If every object in $\eff\EE$ has finite length, then we have $\Ex(\EE) = \AR(\EE)$ by Proposition \ref{part1}, and the condition (CF) is trivially satisfied.

(2) $\Rightarrow$ (3): Obvious.

(3) $\Rightarrow$ (1):
Let $M$ be in $\eff\EE$ and take a conflation $0 \to X \xrightarrow{f} Y \xrightarrow{g} Z \to 0$ in $\EE$ such that $0 \to \EE(-,X) \xrightarrow{\EE(-,f)} \EE(-,Y) \xrightarrow{\EE(-,g)} \EE(-,Z) \to M \to 0$ is exact.
By $\AR(\EE) \otimes_\Z \Q = \Ex(\EE) \otimes_\Z \Q$, there exist a positive integer $a$, integer $b_i$ and an AR conflation $0 \to X_i \to Y_i \xrightarrow{g_i} Z_i \to 0$ for each $i$ such that we have an equality in $\KKK_0(\EE,0)$:
\[
a([X]-[Y]+[Z]) = \sum_{i=1}^n b_i ([X_i] - [Y_i] + [Z_i]).
\]
We may assume that $b_i > 0$ for $i\leq m$ and $-c_i:= b_i <0$ for $i > m$. Thus we have an equality
\begin{equation}\label{eq}
a([X]-[Y]+[Z]) + \sum_{i=m+1}^n c_i([X_i]-[Y_i]+[Z_i]) = \sum_{i=1}^m b_i ([X_i] - [Y_i] + [Z_i])
\end{equation}
in $\KKK_0(\EE,0)$ such that all the coefficients are positive integers. Thus each side of the equation (\ref{eq}) comes from the single conflation in $\EE$.
Put $S_i:= \coker \EE(-,g_i) \in \mod\EE$ for each $i$, which is simple $\EE$-module contained in $\eff\EE$. Then the left hand side in (\ref{eq}) corresponds to $M^a \oplus \bigoplus_{i=m+1}^n S_i^{c_i}$ in $\eff\EE$ and the right hand side corresponds to $\bigoplus_{i=1}^m S_i^{b_i}$ in $\eff\EE$.
Since the latter has finite length, the condition (CF$'$) implies that so does the former. Thus $M$ has finite length.
\end{proof}

\section{Applications}\label{sec4}
In Section \ref{sec3}, we have shown the equivalence of (a), (b) and (c) under three technical assumptions: Hom-finiteness of $\un{\EE}$, the existence of a weak cogenerator of $\un{\EE}$, and the condition (CF) (see Figure \ref{fig} for a summary of these results). In this section, we show that these assumptions are satisfied for some concrete cases.

Main applications we have in mind are the representation theory of noetherian algebras, where the base ring is not necessarily artinian. For the convenience of the reader, we recall some related definitions.
\begin{definition}
Let $R$ be a noetherian local ring and $\Lambda$ an $R$-algebra.
\begin{enumerate}
\item $\Lambda$ is called a \emph{noetherian $R$-algebra} if $\Lambda$ is finitely generated as an $R$-module. If in addition $R$ is artinian, then we say that $\Lambda$ is an \emph{artin $R$-algebra}.
\item Suppose that $R$ is Cohen-Macaulay. For a noetherian $R$-algebra, we denote by $\CM\Lambda$ the category of finitely generated $\Lambda$-modules which are maximal Cohen-Macaulay as $R$-modules.
\item A noetherian $R$-algebra $\Lambda$ is called an \emph{$R$-order} if $\Lambda \in \CM\Lambda$ holds.
\item We say that an $R$-order $\Lambda$ \emph{has at most an isolated singularity} if $\gl \Lambda_p = \height p$ holds for every non-maximal prime $p \in \Spec R$.
\item An $R$-order is called a \emph{Gorenstein} order if $\CM\Lambda$ is a Frobenius exact category.
\end{enumerate}
\end{definition}
Let $R$ be a complete Cohen-Macaulay local ring and $\Lambda$ an $R$-order. Then $\CM\Lambda$ is closed under extensions in $\mod\Lambda$, thus is an exact category with a projective generator $\Lambda$. Moreover, $\CM\Lambda$ has AR conflations if and only if $\Lambda$ has at most an isolated singularity \cite{isolated}.

\subsection{Hom-finiteness of the stable category}
First we consider the Hom-finiteness of $\un{\EE}$. Recall that (a) $\Rightarrow$ (b) holds if $\un{\EE}$ is Hom-finite (Proposition \ref{propab}).
This condition is trivial if $R$ is artinian, hence we obtain the following.
\begin{corollary}
Let $\EE$ be a Hom-finite idempotent complete exact $R$-category over a commutative artinian ring $R$. Suppose that $\EE$ is of finite type. Then $\EE$ is admissible and $\AR(\EE) = \Ex(\EE)$ holds.
\end{corollary}
For an artin algebra $\Lambda$, we can apply this result to any subcategory $\EE$ of $\mod\Lambda$ which is closed under extensions and direct summands.

If $R$ is not artinian, we can use the following fact on Hom-finiteness.
\begin{lemma}[{\cite[Proposition A.5]{en2}}]\label{homfin}
Let $R$ be a complete noetherian local ring and $\EE$ a Hom-noetherian idempotent complete exact $R$-category. Suppose that $\EE$ has enough projectives and consider the following conditions:
\begin{enumerate}
\item $\EE$ is of finite type.
\item $\EE$ has AR conflations.
\item $\un{\EE}$ is Hom-finite over $R$.
\end{enumerate}
Then the implication {\upshape (1)} $\Rightarrow$ {\upshape (2)} $\Rightarrow$ {\upshape (3)} holds.
\end{lemma}
Thus we immediately obtain the following.
\begin{corollary}[{\cite[Corollary 3.18]{en2}}]\label{finadm}
Let $R$ be a complete noetherian local ring and $\EE$ a Hom-noetherian idempotent complete exact $R$-category. Suppose that $\EE$ has enough projectives and $\EE$ is of finite type. Then $\EE$ is admissible and $\AR(\EE) = \Ex(\EE)$ holds.
\end{corollary}
Note that the assumption on $\EE$ here is rather mild. In particular, this corollary can be applied to $\EE:= \CM\Lambda$ for an $R$-order $\Lambda$ which is \emph{CM-finite}, that is, $\CM\Lambda$ is of finite type. Thus this provides a further generalization of \cite[Proposition 2.2]{relations}.

\subsection{Existence of a weak cogenerator of the stable category}
For the implication (b) $\Rightarrow$ (a), we need a weak cogenerator of $\un{\EE}$ (Proposition \ref{propab}). To this purpose, let us recall the notions of functorially finiteness and resolving subcategories.
\begin{definition}\label{functfin}
Let $\AA$ be an additive category and $\EE$ an additive subcategory of $\AA$.
\begin{enumerate}
 \item A morphism $f:E_X \to X$ in $\AA$ is said to be a \emph{right $\EE$-approximation} if $E_X$ is in $\EE$ and every morphism $E \to X$ with $E\in\DD$ factors through $f$.
 \item $\EE$ is said to be \emph{contravariantly finite} if every object in $\AA$ has a right $\EE$-approximation.
\end{enumerate}
A \emph{left $\EE$-approximation} and \emph{covariantly finiteness} are defined dually. We say that $\EE$ is \emph{functorially finite} if it is both contravariantly and covariantly finite.
\end{definition}

\begin{definition}
Let $\Lambda$ be a noetherian ring and $\EE$ a subcategory of $\mod\Lambda$. We say that $\EE$ is a \emph{resolving} subcategory of $\mod\Lambda$ if it satisfies the following conditions:
\begin{enumerate}
 \item $\EE$ contains $\Lambda$.
 \item $\EE$ is closed under extensions, that is, for each exact sequence $0 \to X \to Y \to Z \to 0$, if both $X$ and $Z$ are in $\EE$, then so is $Y$.
 \item $\EE$ is closed under kernels of surjections, that is, for each exact sequence $0 \to X \to Y \to Z \to 0$, if both $Y$ and $Z$ are in $\EE$, then so is $X$.
 \item $\EE$ is closed under summands.
 \end{enumerate}
\end{definition}
Since each resolving subcategory is an extension-closed subcategory of an abelian category, it has the natural exact structure. Thereby we always regard it as an exact category.

The following result provides a rich source of exact categories such that its stable category has a weak cogenerator.
\begin{proposition}\label{weakcogen}
Let $\Lambda$ be a semiperfect noetherian ring and $\EE$ a contravariantly finite resolving subcategory of $\mod\Lambda$. Then $\un{\EE}$ has a weak cogenerator, which can be chosen as a right $\EE$-approximation of $\Lambda / \rad \Lambda$.
\end{proposition}
\begin{proof}
Since $\EE$ is resolving in $\mod\Lambda$, it is straightforward to see that $\EE$ has a projective generator $\Lambda$.
Let us construct a weak cogenerator $\un{X}$ of $\un{\EE}$ (cf. \cite[Lemma 2.4]{ar1}).
Since $\Lambda$ is semiperfect, $\Lambda / \rad \Lambda$ is a direct sum of all simple $\Lambda$-modules up to multiplicity. Take a right $\EE$-approximation $f:X \to \Lambda / \rad \Lambda$ of $\Lambda / \rad \Lambda$, and we claim that $\un{X}$ is a weak cogenerator of $\un{\EE}$.

Suppose that $\un{\EE}(\un{W},\un{X})=0$. Then it follows from this that every morphism from $W$ to finitely generated semisimple modules factors through some projective module. We will show that $W$ is projective, that is, $\un{W}=0$ in $\un{\EE}$.

For a Jacobson radical $\rad W$ of $W$, let $\pi:W \defl W/\rad W$ be a natural projection and take a projective cover $p:P \defl W/\rad W$ (this is possible since $\Lambda$ is semiperfect). It follows that there exists a projective cover $\varphi:P \defl W$ such that the following diagram commutes:

\[
\begin{tikzcd}
P \arrow[r,"\varphi", two heads] \arrow[rd, "p"', two heads] & W \arrow[d,"\pi", two heads]\arrow["\psi", r, dashed] & P \arrow[ld, "p", two heads] \\
 & W/\rad W
\end{tikzcd}
\]
Since $W/\rad W$ is semisimple, $\pi$ must factor through some projective module. Thus there exists $\psi:W \to P$ which makes the above diagram commute. On the other hand, since $p$ is right minimal, $\psi \circ \varphi$ must be an isomorphism. It follows that $\varphi$ is an isomorphism, thus $W$ is projective.
\end{proof}

Now one immediately obtains the following result about the equivalence of (a) and (b) for a contravariantly finite resolving subcategories over noetherian algebras. We will treat the case of artin algebras later (see Theorem \ref{artinmain}) since one can prove more.
\begin{corollary}\label{cmab}
Let $R$ be a complete noetherian local ring, $\Lambda$ a noetherian $R$-algebra and $\EE$ a contravariantly finite resolving subcategory of $\mod\Lambda$. Then $\EE$ is of finite type if and only if $\EE$ is admissible.
\end{corollary}
\begin{proof}
There exists a weak cogenerator of $\un{\EE}$ by Proposition \ref{weakcogen} since $\Lambda$ is semiperfect. Thus $\EE$ is of finite type if $\EE$ is admissible by Proposition \ref{propab}.
The other implication follows from Corollary \ref{finadm}.
\end{proof}

\begin{example}\label{cmcontfin}
Let $R$ be a complete Cohen-Macaulay local ring and $\Lambda$ an $R$-order. Then it is well-known that $\CM\Lambda$ is a contravariantly finite resolving subcategory of $\mod\Lambda$ (e.g. \cite{ab}). In particular, $\un{\CM}\,\Lambda$ has a weak cogenerator by Proposition \ref{weakcogen} and Corollary \ref{cmab} applies to $\CM\Lambda$.
Thus Corollary \ref{cmab} generalizes the result \cite[Lemma 2.4 (b)]{relations} on the category $\CM\Lambda$.
\end{example}

\subsection{On the condition (CF)}
Recall that in Section \ref{secbc}, we have investigated the relation between (b) admissibility and (c) AR=Ex, and Theorem \ref{fcfl} shows that the condition (CF) ensures that (b) and (c) are equivalent. Now let us give some actual examples where $\EE$ satisfies (CF).
\emph{In this subsection, we fix a commutative noetherian ring $R$}.

First we deal with exact categories which are Hom-finite.
\begin{proposition}\label{homfinfc}
Let $\EE$ be a Hom-finite exact $R$-category. Then $\EE$ satisfies {\upshape (CF)}.
\end{proposition}
\begin{proof}
Let $0 \to X \xrightarrow{f_i} Y \xrightarrow{g_i} Z \to 0$ be conflations in $\EE$ for $i=1,2$, and put $M_i := \coker \EE(-,g_i) \in \eff\EE$. Then we have that $[M_1]= [\EE(-,X)] - [\EE(-,Y)] + [\EE(-,Z)] = [M_2]$ holds in $\KKK_0(\mod\EE)$. We will show that if $M_1$ has finite length, then so does $M_2$. It suffices to show that $\supp M_2$ is finite by Lemma \ref{fl}.

Let $X$ be an indecomposable object in $\EE$. For each $F$ in $\mod\EE$, the assignment $F \mapsto \length_R F(X)$ makes sense since $\EE$ is Hom-finite. It clearly extends to the group homomorphism $\chi_X: \KKK_0(\mod\EE) \to \Z$. Now  we have that $X\in\supp M_1 \Leftrightarrow \chi_X [M_1] \neq 0 \Leftrightarrow \chi_X [M_2] \neq 0 \Leftrightarrow X\in\supp M_2$ by $[M_1] = [M_2]$. Thus $\supp M_1 = \supp M_2$ holds, and this is finite since $M_1$ has finite length.
\end{proof}

Now we immediately obtain the following general result on artin algebras.
\begin{theorem}\label{artinmain}
Let $\Lambda$ be an artin algebra and $\EE$ a contravariantly finite resolving subcategory of $\mod\Lambda$. Then the following are equivalent.
\begin{enumerate}
\item $\EE$ is of finite type.
\item $\EE$ is admissible.
\item $\AR(\EE) = \Ex(\EE)$ holds.
\item $\AR(\EE) \otimes_\Z \Q = \Ex(\EE) \otimes_\Z \Q$ holds.
\end{enumerate}
\end{theorem}
\begin{proof}
Since $\Lambda$ is artin algebra, $\Lambda$ is semiperfect and $\un{\EE}$ is Hom-finite. Thus Propositions \ref{propab} and \ref{weakcogen} imply that (1) and (2) are equivalent. Furthermore, since $\EE$ is Hom-finite, $\EE$ satisfies (CF) by Proposition \ref{homfinfc}. Thus the other conditions are also equivalent by Theorem \ref{fcfl}.
\end{proof}

Contravariantly finite resolving subcategories are closely related to so called \emph{cotilting modules} by the famous result of \cite{applications}. For the convenience of the reader, we explain their result here.
\begin{definition}
Let $\Lambda$ be an artin algebra and $U$ a $\Lambda$-module in $\mod\Lambda$. Then $U$ is called a \emph{cotilting $\Lambda$-module} if it satisfies the following conditions:
\begin{enumerate}
\item the injective dimension of $U$ is finite.
\item $\Ext_\Lambda^{>0}(U,U)=0$ holds.
\item There exists an exact sequence $0 \to U_n \to \cdots \to U_0 \to D \Lambda \to 0$ of $\Lambda$-modules with $U_i \in \add U$ for each $i$, where $D$ denotes the standard duality $D: \mod \Lambda^{\op} \to \mod\Lambda$.
\end{enumerate}
\end{definition}
For a module $U$ in $\mod\Lambda$, we denote by $^\perp U$ the full subcategory of $\mod\Lambda$ consisting of modules $X$ satisfying $\Ext^{>0}_\Lambda(X,U)=0$.
\begin{proposition}[{\cite[Theorem 5.5]{applications}}]\label{cotiltcont}
Let $\Lambda$ be an artin algebra and $U$ a cotilting $\Lambda$-module. Then $^\perp U$ is a contravariantly finite resolving subcategory of $\mod\Lambda$.
\end{proposition}
Thus we can apply Theorem \ref{artinmain} to $\EE:= {}^\perp U$ for a cotilting module $U$ over an artin algebra. This provides a rich source of examples:
\begin{itemize}
\item Let $\EE$ be a functorially finite torsionfree class of $\mod\Lambda$ for an artin algebra $\Lambda$ (see \cite{ASS}). Then by factoring out the annihilator of $\EE$, we may assume that $\EE$ is a faithful functorially finite torsionfree class. It is well-known that such $\EE$ is of the form $^\perp U$ for a cotilting $\Lambda$-module with $\id U \leq 1$ (see e.g. \cite[Theorem]{smalo}), so Theorem \ref{artinmain} applies to $\EE$. By duality, every functorially finite torsion class over artin algebras is also an example.
\item An artin algebra $\Lambda$ is called \emph{Iwanaga-Gorenstein} if both $\id(\Lambda_\Lambda)$ and $\id({}_\Lambda \Lambda)$ are finite. In this case, we write $\GP\Lambda := {}^\perp \Lambda$ and call it the category of \emph{Gorenstein-projective} modules. It is immediate that $\Lambda$ itself is a cotilting $\Lambda$-module, so we can apply Theorem \ref{artinmain} to $\EE := \GP\Lambda$.
\end{itemize}

Next we will consider the case $\dim R > 0$.
Although $\EE$ is rarely Hom-finite, we can conclude (CF) by using the Hom-finiteness of the stable category $\un{\EE}$ in some cases. Let us introduce some terminologies to state our result.

Let $\EE$ be an exact category with enough projectives. In this case, we have the syzygy functor $\Omega: \un{\EE} \to \un{\EE}$. We denote by $\un{X}$ in $\un{\EE}$ for the image of $X$ in $\EE$ under the natural projection $\EE \defl \un{\EE}$. For an object $X$ in $\EE$, we say that $X$ has \emph{finite projective dimension} if there exists some integer $n\geq 0$ such that $\Omega^n \un{X} = 0$ holds in $\un{\EE}$.

Recall that an exact category $\EE$ is called \emph{Frobenius} if $\EE$ has enough projectives and enough injectives, and the classes of projectives and injectives coincide.

\begin{proposition}\label{glfc}
Let $R$ be a complete noetherian local ring and $\EE$ a Hom-noetherian idempotent complete exact $R$-category with AR conflations. Suppose that $\EE$ satisfies either one of the following conditions:
\begin{enumerate}
\item $\EE$ has enough projectives and every object in $\EE$ has finite projective dimension, or
\item $\EE$ is Frobenius.
\end{enumerate}
 Then $\EE$ satisfies {\upshape (CF)}.
\end{proposition}
\begin{proof}
Let $0 \to X \xrightarrow{f_i} Y \xrightarrow{g_i} Z \to 0$ be conflations in $\EE$ for $i=1,2$. Put $M_i := \coker \EE(-,g_i) \in \eff\EE$ for each $i$ and suppose that $M_1$ has finite length. We will show that so does $M_2$. First note that Lemma \ref{effstable} applies to this situation since $\EE$ has enough projectives, so we identify $\eff\EE$ with $\mod\un{\EE}$ naturally. Then $M_i$ has finite length as an $\EE$-module if and only if it does so as an $\un{\EE}$-module. Moreover, since $\un{\EE}$ is Hom-finite by Lemma \ref{homfin}, this occurs if and only if $\supp M_i$ is finite, where the support is considered inside $\ind \un{\EE}$.

(1)
First we prove that $[M_1]=[M_2]$ holds in $\KKK_0(\mod\un{\EE})$. Choose an integer $n \geq 0$ such that $\Omega^n \un{Z} = 0$ holds in $\un{\EE}$. Then it is classical that we have an exact sequence in $\mod\un{\EE}$ for each $i$:
\begin{align*}
0 &\to \un{\EE}(-,\Omega^{n-1}\un{X}) \to \un{\EE}(-,\Omega^{n-1}\un{Y}) \to \un{\EE}(-,\Omega^{n-1}\un{Z}) \to \cdots \\
&\to \un{\EE}(-,\Omega \un{X}) \to \un{\EE}(-,\Omega \un{Y}) \to \un{\EE}(-,\Omega \un{Z}) \to \un{\EE}(-,\un{X}) \to \un{\EE}(-,\un{Y}) \to \un{\EE}(-,\un{Z}) \to M_i \to 0.
\end{align*}
Therefore it immediately follows that $[M_1] = [M_2]$ holds in $\KKK_0(\mod\un{\EE})$.

Let $\un{X}$ be an indecomposable object of $\un{\EE}$. Since $\un{\EE}$ is a Hom-finite $R$-category by Lemma \ref{homfin}, we have the characteristic map $\chi_{\un{X}}:\KKK_0(\mod\un{\EE}) \to \Z$ which sends $F$ to $\length_R F(\un{X})$. Then we have $\un{X} \in \supp M_1 \Leftrightarrow \chi_{\un{X}}[M_1] \neq 0 \Leftrightarrow \chi_{\un{X}}[M_2] \neq 0 \Leftrightarrow \un{X} \in \supp M_2$, where the support is considered inside $\ind \un{\EE}$. Thus Proposition \ref{fl} implies that $M_2$ has finite length.

(2)
Put $N_i := \ker \un{\EE}(-,\un{f_i})$ for each $i=1,2$. Then we have exact sequences in $\mod\un{\EE}$ as in (1):
\begin{align*}
0 \to N_i \to \un{\EE}(-,\un{X}) \xrightarrow{\un{\EE}(-,\un{f_i})} \un{\EE}(-,\un{Y}) &\xrightarrow{\un{\EE}(-,\un{g_i})} \un{\EE}(-,\un{Z}) \to M_i \to 0, \\
\un{\EE}(-,\Omega\un{Y}) &\xrightarrow{\un{\EE}(-,\Omega\un{g_i})} \un{\EE}(-,\Omega\un{Z}) \to N_i \to 0.
\end{align*}
It follows from the first exact sequence that $[M_1\oplus N_1] = [M_1] + [N_1] = [M_2] + [N_2] = [M_2\oplus N_2]$ holds in $\KKK_0(\mod\un{\EE})$. We will show that $N_1$ has finite length. By the same argument as in (1), we have that $\supp (M_1\oplus N_1) = \supp (M_2 \oplus N_2)$ holds in $\ind\un{\EE}$.
Now $\Omega : \un{\EE} \to \un{\EE}$ is an equivalence since $\EE$ is Frobenius, and let us denote by $\Omega^-:\un{\EE} \to \un{\EE}$ the quasi-inverse of $\Omega$. Then it is easily checked that $\un{X} \in \supp N_1 \Leftrightarrow \Omega^- \un{X} \in \supp M_1 \Leftrightarrow \un{X} \in \Omega (\supp M_1)$ holds for an indecomposable object $\un{X} \in \un{\EE}$, hence $\supp N_1 = \Omega (\supp M_1)$. Therefore $\supp N_1$ is finite, and so is $\supp (M_1\oplus N_1) = \supp (M_2 \oplus N_2)$. Hence $\supp M_2$ is finite, which implies that $M_2$ has finite length.
\end{proof}

For an $R$-order $\Lambda$ with at most an isolated singularity, Proposition \ref{glfc} shows that $\CM\Lambda$ satisfies (CF) if either $\Lambda$ has finite global dimension or $\Lambda$ is a Gorenstein order. Thus we obtain the following result on the category $\CM\Lambda$.
\begin{theorem}\label{ordermain}
Let $R$ be a complete Cohen-Macaulay local ring and $\Lambda$ an $R$-order with at most an isolated singularity. Put $\EE := \CM\Lambda$. Then the following are equivalent.
\begin{enumerate}
\item $\EE$ is of finite type.
\item $\EE$ is admissible.
\end{enumerate}
Assume in addition that either $\Lambda$ has finite global dimension or $\Lambda$ is a Gorenstein order. Then the following are also equivalent.
\begin{enumerate}[resume]
\item $\AR(\EE) = \Ex(\EE)$ holds.
\item $\AR(\EE) \otimes_\Z \Q = \Ex(\EE) \otimes_\Z \Q$ holds.
\end{enumerate}
\end{theorem}
\begin{proof}
This follows from Example \ref{cmcontfin} and Proposition \ref{glfc}.
\end{proof}
\begin{remark}
Theorem \ref{ordermain} for the case $\Lambda$ has finite global dimension was shown in \cite[Proposition 2.3]{ar1}, but their proof relies on higher algebraic K-theory. Also Theorem \ref{ordermain} generalizes \cite{hir} where $\Lambda$ is assumed to be commutative and Gorenstein.
The author does not know whether all the conditions above are equivalent without any assumption on an $R$-order $\Lambda$, even if $\Lambda$ is commutative.
\end{remark}

\section{Finiteness of syzygies}\label{sec5}
Let $R$ be a complete Cohen-Macaulay local ring and suppose that $\AR(\CM R) = \Ex(\CM R)$ holds. Although we do not know whether $\CM R$ is of finite type, it is shown in \cite{ko} that $\mathsf{\Omega CM}\,R$, the category of syzygies of $\CM R$ is so. In this section, we will extend this result to a non-commutative order.

Let $\EE$ be an exact category with enough projectives. In this section, we denote by $\Omega\EE$ the subcategory of $\EE$ consisting of objects $X$ such that there exists an inflation $X \infl P$ in $\EE$ for some projective object $P$. If $\EE$ is Krull-Schmidt, then so is $\Omega \EE$ since $\Omega\EE$ is closed under direct sums and summands. The essential image of $\Omega\EE$ under the natural projection $\EE \defl \un{\EE}$ coincides with $\Omega \un{\EE}$, the essential image of the syzygy functor $\Omega:\un{\EE} \to \un{\EE}$.
We begin with the following property about $\Omega\EE$, which is of interest in itself.

\begin{lemma}\label{omega-}
Let $R$ be a commutative noetherian ring, and let $\EE$ be a Hom-noetherian idempotent complete exact $R$-category with a projective generator $P$. Then the syzygy functor $\Omega:\un{\EE} \to \Omega\un{\EE}$ has a fully faithful left adjoint $\Omega^-: \Omega\un{\EE}\to \un{\EE}$.
\end{lemma}
\begin{proof}
For each object $W\in\Omega \EE$, take a left $(\add P)$-approximation $f:W \to P^W$ (this is possible since $\EE$ is Hom-noetherian). Since $W$ is in $\Omega\EE$, there exists an inflation $f':W \infl P'$ with $P'\in \add P$. Thus $f'$ factors through $f$, which implies that $f$ is an inflation by the idempotent completeness of $\EE$ (\cite[Proposition 7.6]{buhler}).
Now we define $\Omega^- W\in\EE$ by the following conflation:
\[
0 \to W \xrightarrow{f} P^W \xrightarrow{g} \Omega^- W \to 0.
\]

Next suppose that we have a map $\varphi:W_1 \to W_2$ in $\Omega\EE$. This induces the commutative diagram
\[
\begin{tikzcd}
0 \arrow[r] & W_1 \arrow[r,"f_1"] \arrow[d,"\varphi"] & P^{W_1} \arrow[r,"g_1"] \arrow[d,"\psi"] & \Omega^- W_1 \arrow[r]\arrow[d,"\Omega^- \varphi"] & 0\\
0 \arrow[r] & W_2 \arrow[r,"f_2"] & P^{W_2} \arrow[r,"g_2"] & \Omega^- W_2 \arrow[r]& 0
\end{tikzcd}
\]
in $\EE$ where both rows are conflations, since $f_2 \circ \varphi$ must factor through $f_1$. A simple diagram chase shows that this defines a well-defined isomorphism $\Omega^-:\un{\EE}(W_1,W_2) \to \un{\EE}(\Omega^- W_1, \Omega^- W_2)$, with its inverse induced by the syzygy functor $\Omega:\un{\EE} \to \Omega \un{\EE}$. Thus we obtain a functor $\Omega^-: \Omega\un{\EE} \to \un{\EE}$, which is fully faithful and satisfies $\Omega \Omega^- \iso \id_{\Omega \un{\EE}}$.

Finally we show that $\Omega^-:\Omega\un{\EE} \to \un{\EE}$ is a left adjoint of $\Omega:\un{\EE} \to \Omega\un{\EE}$. Let $W$ be in $\Omega\EE$ and $X$ in $\EE$. Then the syzygy functor induces a morphism $\un{\EE}(\Omega^- W,X) \to \un{\EE}(W,\Omega X)$ by $\Omega \Omega^- W \iso W$ in $\un{\EE}$. It is easy to show that this map is bijective, by considering the following diagram
\[
\begin{tikzcd}
0 \arrow[r] & W \arrow[r,"f"] \arrow[d] & P^W \arrow[r,"g"] \arrow[d] & \Omega^- W \arrow[r]\arrow[d] & 0\\
0 \arrow[r] & \Omega X \arrow[r]& P_X \arrow[r]& X \arrow[r]& 0
\end{tikzcd}
\]
where $P_X$ is projective and $f$ is a left $\add P$-approximation. The details are left to the reader.
\end{proof}

\begin{theorem}\label{omega}
Let $R$ be a complete noetherian local ring and $\EE$ a Hom-noetherian idempotent complete exact $R$-category with a projective generator $P$ and AR conflations. Suppose that $\AR(\EE)\otimes_\Z \Q = \Ex(\EE) \otimes_\Z \Q$ holds. Then the following holds.
\begin{enumerate}
\item For every object $M$ in $\eff\EE$, we have that $\supp M \cap \ind (\Omega\un{\EE})$ is finite, where the support is considered inside $\ind\un{\EE}$.
\item Assume in addition that $\un{\EE}$ has a weak cogenerator. Then $\Omega \EE$ is of finite type.
\end{enumerate}
\end{theorem}

\begin{proof}
Throughout this proof, supports of functors are always considered inside $\ind\un{\EE}$.

(1)
First recall that the stable category $\un{\EE}$ is Hom-finite over $R$ by Lemma \ref{homfin}.
By  the similar argument as in Proposition \ref{fcfl}, it suffices to prove the following weaker version of (CF).

\begin{itemize}
\item[(CF)$_\Omega$] Let $0 \to X \xrightarrow{f_i} Y \xrightarrow{g_i} Z \to 0$ be conflations in $\EE$ for $i=1,2$ and put $M_i := \coker \EE(-,g_i) \in \eff\EE$. If $\supp M_1$ is finite, then $\supp M_2 \cap \ind(\Omega\un{\EE})$ is finite.
\end{itemize}
Assume the above situation and put $N_i := \ker \un{\EE}(-,\un{f_i})$ for each $i=1,2$. Then we have exact sequences in $\mod\un{\EE}$ for $i=1,2$, as in the proof of Proposition \ref{glfc}(2):
\begin{align*}
0 \to N_i \to \un{\EE}(-,\un{X}) \xrightarrow{\un{\EE}(-,\un{f_i})} \un{\EE}(-,\un{Y}) &\xrightarrow{\un{\EE}(-,\un{g_i})} \un{\EE}(-,\un{Z}) \to M_i \to 0, \\
\un{\EE}(-,\Omega\un{Y}) &\xrightarrow{\un{\EE}(-,\Omega \un{g_i})} \un{\EE}(-,\Omega\un{Z}) \to N_i \to 0.
\end{align*}
As in the proof of Proposition \ref{glfc}(2), it is enough to show that $\supp N_1 \cap \ind(\Omega\un{\EE})$ is finite. Actually we will show that every element of $\supp N_1 \cap \ind(\Omega\un{\EE})$ is isomorphic to $\Omega\un{A}$ for some $\un{A} \in \supp M_1$, which obviously implies the desired claim.

Let $\un{W}$ be an element of $\supp N_1 \cap \ind(\Omega\un{\EE})$, so $\un{\EE}(\un{W},\Omega \un{g_1}): \un{\EE}(\un{W},\Omega \un{Y_1}) \to \un{\EE}(\un{W},\Omega \un{Z_1})$ is not surjective. By Lemma \ref{omega-}, this is equivalent to that $\un{\EE}(\Omega^- \un{W},\un{g_1}) : \un{\EE}(\Omega^-\un{W},\un{Y_1}) \to \un{\EE}(\Omega^- \un{W}, \un{Z_1})$ is not surjective. On the other hand, since $\Omega^- :\Omega\EE \to \EE$ is fully faithful by Lemma \ref{omega-}, we have that $\Omega^- \un{W}$ is indecomposable. Thus $\Omega^- \un{W}$ belongs to $\supp M_1$. Now $\un{W} \iso \Omega \Omega^- \un{W}$ in $\un{\EE}$ holds, which completes the proof.

(2)
Let $\un{X}$ be a weak cogenerator of $\un{\EE}$. By (1), we have that that $\ind (\Omega\un{\EE}) \cap \supp \un{\EE}(-,\un{X})$ is finite. Since every indecomposable object $\un{W}$ in $\un{\EE}$ should satisfy  $\un{\EE}(\un{W},\un{X})\neq 0$, it follows that $\Omega\un{\EE}$ is of finite type. Now $\Omega\EE$ itself is of finite type because we have a natural identification $\ind(\Omega\EE) = \ind(\Omega \un{\EE}) \sqcup \ind (\add P)$.
\end{proof}
Now we apply this theorem to the category of Cohen-Macaulay modules. The obtained result extends \cite{hir,ko}, where $\Lambda$ was assumed to be commutative.
\begin{corollary}\label{syzygymain}
Let $R$ be a complete Cohen-Macaulay local ring and $\Lambda$ an $R$-order with at most an isolated singularity. If $\AR(\CM \Lambda)\otimes_\Z \Q = \Ex(\CM \Lambda) \otimes_\Z \Q$ holds, then $\mathsf{\Omega CM}\,\Lambda$ is of finite type.
\end{corollary}
\begin{proof}
We can apply Theorem \ref{omega} to $\EE:= \CM\Lambda$ since $\un{\EE}$ has a weak cogenerator by Example \ref{cmcontfin}.
\end{proof}

\appendix
\section{Some properties of the category of effaceable functors}
First we show that our notion of effaceability is equivalent to that of Grothendieck (see \cite[p. 148]{gro} or \cite[Exercise 2.4.5]{weibel} for the standard definition).

\begin{proposition}\label{grodef}
Let $\EE$ be an idempotent complete exact category and $M \in \Mod\EE$ a right $\EE$-module. Then the following are equivalent.
\begin{enumerate}
\item $M$ is in $\eff \EE$.
\item $M$ is finitely presented, and for every $W \in \EE$ and $w\in M(W)$, there exists a deflation $\psi: E \defl W$ in $\EE$ such that $(M\psi )(w)=0$.
\end{enumerate}
\end{proposition}
\begin{proof}
(1) $\Rightarrow$ (2):
Suppose that $M$ is in $\eff\EE$ and take a conflation $0 \to X \to Y \xrightarrow{g} Z \to 0$ in $\EE$ such that $M \iso \coker \EE(-,g)$. Denote by $\pi:\EE(-,Z) \defl M$ the natural projection. Clearly $M$ is finitely presented. Let $W\in\EE$ and $w \in M(W)$. By the Yoneda lemma, we have a morphism $w_* :\EE(-,W) \to M$. By the projectivity of $\EE(-,W)$ and the Yoneda lemma, we have a map $\varphi : W \to Z$ such that $\pi \circ \EE(-,\varphi) = w_*$.
Now let us take the pullback of $g$ along $\varphi$:
\[
\begin{tikzcd}
0 \arrow[r]& X \arrow[d,equal] \arrow[r]& E \arrow[d] \arrow[r,"\psi"] & W \arrow[r] \arrow[d,"\varphi"] & 0 \\
0 \arrow[r]& X \arrow[r]& Y \arrow[r,"g"] & Z \arrow[r]& 0
\end{tikzcd}
\]
Note that $\psi$ is a deflation. This yields the following commutative diagram in $\Mod\EE$:
\[
\begin{tikzcd}
&  & \EE(-,E) \arrow[r,"{\EE(-,\psi)}"]\arrow[d] & \EE(-,W) \arrow[d,"{\EE(-,\varphi)}"'] \arrow[dr,"w_*"] \\
0 \arrow[r]& \EE(-,X) \arrow[r]& \EE(-,Y) \arrow[r,"{\EE(-,g)}"']& \EE(-,Z) \arrow[r,"\pi"']& M \arrow[r]& 0
\end{tikzcd}
\]
By applying these functors to $E$, it is straightforward to check that $(M\psi )(w)=0$.

(2) $\Rightarrow$ (1):
Suppose that $M$ satisfies the condition of (2). Since $M$ is finitely presented, we have a morphism $g:Y\to Z$ in $\EE$ such that $\EE(-,Y) \to \EE(-,Z) \to M \to 0$ is exact. We have an element $z \in M(Z)$ which corresponds to $\EE(-,Z) \to M$ by the Yoneda lemma. By the assumption, there exists a deflation $\psi: E \defl Z$ such that $(M \psi)(z)=0$. The Yoneda lemma implies that the composition $\EE(-,E) \to \EE(-,Z) \to M $ is a zero map, hence the map $\EE(-,\psi)$ lifts to $\EE(-,Y)$. Thus we have a map $h:E \to Y$ such that $g \circ h = \psi$. Since $\psi$ is a deflation, we can conclude that so is $g$ by the idempotent completeness of $\EE$. Therefore $M$ is in $\eff\EE$.
\end{proof}

Next we prove that $\eff\EE$ is an abelian subcategory of $\Mod\EE$.
\begin{theorem}\label{effabelian}
Let $\EE$ be an exact category. Then $\eff\EE$ is closed under kernels and cokernels in $\Mod\EE$. In particular, $\eff\EE$ is an abelian category such that the inclusion $\eff\EE \to \Mod\EE$ is exact. Moreover, $\eff\EE$ is closed under extensions in $\Mod\EE$.
\end{theorem}
\begin{proof}
Let $\varphi:M_1 \to M_2$ be a morphism in $\eff\EE$. For each $i=1,2$, we have the corresponding conflations $0 \to X_i \xrightarrow{f_i} Y_i \xrightarrow{g_i} Z_i \to 0$ in $\EE$ satisfying $M_i = \coker \EE(-,g_i)$.
We will show that the image, kernel and cokernel of $\varphi$ in $\Mod\EE$ are contained in $\eff\EE$. Sine $\varphi$ induces a morphism between the projective resolutions of $M_1$ and $M_2$, we have the following commutative diagram in $\EE$ by the Yoneda lemma:
\[
\begin{tikzcd}
0 \arrow[r]& X_1 \arrow[r,"f_1"] \arrow[d,"a"] & Y_1 \arrow[r,"g_1"] \arrow[d,"b"] & Z_1 \arrow[r] \arrow[d,"c"] & 0 \\
0 \arrow[r]& X_2 \arrow[r,"f_2"] & Y_2 \arrow[r,"g_2"] & Z_2 \arrow[r] & 0
\end{tikzcd}
\]
Take the pushout $E$ of the top conflation along $a:X_1 \to X_2$. By the universal property, we obtain the following commutative diagram
\begin{equation}\label{kore}
\begin{tikzcd}
0 \arrow[r]& X_1 \arrow[d,equal] \arrow[r, "{^t [-a,f_1]}"] & X_2 \oplus Y_1 \arrow[r,"{[f,b_1]}"] \arrow[d,"{[0,1]}"] & E \arrow[r]\arrow[d,"g"] & 0\\
0 \arrow[r]& X_1 \arrow[r,"f_1"] \arrow[d,"a"] & Y_1 \arrow[r,"g_1"] \arrow[d,"b_1"] & Z_1 \arrow[r] \arrow[d,equal] & 0 \\
0 \arrow[r]& X_2 \arrow[r,"f"] \arrow[d,equal] & E \arrow[r,"g"] \arrow[d,"b_2"] & Z_1 \arrow[r]\arrow[d,"c"] & 0 \\
0 \arrow[r]& X_2 \arrow[r,"f_2"]\arrow[d,"f"] & Y_2 \arrow[r,"g_2"]\arrow[d,"{^t [1,0]}"] & Z_2 \arrow[r] \arrow[d,equal] & 0\\
0 \arrow[r]& E \arrow[r,"{^t [b_2,g]}"] & Y_2 \oplus Z_1 \arrow[r,"{[g_2,-c]}"] & Z_2 \arrow[r]& 0
\end{tikzcd}
\end{equation}
in $\EE$ such that $b_2 \circ b_1 = b$ holds. It is straightforward to see that diagram (\ref{kore}) commutes. Moreover, all the rows are conflations, and the top-right square and the bottom-left square are pullback-pushout squares, see \cite[Proposition 2.12]{buhler}

By the Yoneda embedding, we obtain the following commutative diagram in $\Mod\EE$ corresponding to (\ref{kore}), where all the rows are exact.

\[
\begin{tikzcd}
0 \arrow[r] & \EE(-,X_1) \arrow[r]\arrow[equal,d]& \EE(-,X_2 \oplus Y_1) \arrow[r] \arrow[d] & \EE(-,E) \arrow[r] \arrow[d] & K \arrow[r]\arrow[d,"\iota"] & 0 \\
0 \arrow[r]& \EE(-,X_1) \arrow[r] \arrow[d] & \EE(-,Y_1) \arrow[r]\arrow[d] & \EE(-,Z_1) \arrow[r] \arrow[equal,d] & M_1 \arrow[r]\arrow[d,"\varphi_1"] & 0 \\
0 \arrow[r]& \EE(-,X_2) \arrow[r] \arrow[equal,d] & \EE(-,E) \arrow[r]\arrow[d] & \EE(-,Z_1) \arrow[r] \arrow[d] & M \arrow[r] \arrow[d,"\varphi_2"] & 0 \\
0 \arrow[r]& \EE(-,X_2) \arrow[r]\arrow[d] & \EE(-,Y_2) \arrow[r]\arrow[d] & \EE(-,Z_2) \arrow[r]\arrow[equal,d]  & M_2 \arrow[r]\arrow[d,"\pi"] & 0 \\
0 \arrow[r]& \EE(-,E) \arrow[r]& \EE(-,Y_2 \oplus Z_1) \arrow[r]& \EE(-,Z_2) \arrow[r]& C \arrow[r]& 0
\end{tikzcd}
\]
We have that $K$, $M$ and $N$ belong to $\eff\EE$. We can check that $\varphi_1$ is surjective and $\varphi_2$ is injective, hence $M$ is a image of $\varphi$. Moreover, one can show that $\iota$ is a kernel of $\varphi_1$ (or equivalently, $\varphi$) and $\pi$ is a cokernel of $\varphi_2$ (or equivalently $\varphi$). We leave the details to the reader.

For the extension-closedness of $\eff\EE$, we refer the reader to \cite[Proposition 2.10]{en2}.
\end{proof}

In this paper, we often deal with the finiteness of length of effaceable $\EE$-modules. Since $\eff\EE$ and $\Mod\EE$ are both abelian, the length of effaceable modules seems to depend on the ambient category we adopt. This is not the case by the following.
\begin{proposition}\label{onaji}
Let $\EE$ be an exact category and $M$ an object in $\eff\EE$.
\begin{enumerate}
\item $M$ has finite length in $\eff\EE$ if and only if it does so in $\Mod\EE$.
\item Suppose that $\mod\EE$ is abelian. Then $M$ has finite length in $\Mod\EE$ if and only if it does so in $\mod\EE$.
\end{enumerate}
Moreover, the length and the composition factors of $M$ in $\Mod\EE$ coincide with those in $\eff\EE$ and in $\mod\EE$.
\end{proposition}

This follows from the following elementary observation.
\begin{lemma}\label{elementary}
Let $\EE$ be an additive category and $\AA$ an abelian subcategory of $\Mod\EE$ such that the inclusion $\AA \to \Mod\EE$ is exact. Suppose that $\AA$ is closed under finitely generated submodules. Then the following are equivalent for every object $M$ in $\AA$.
\begin{enumerate}
\item $M$ has finite length in $\Mod\EE$.
\item $M$ has finite length in $\AA$.
\end{enumerate}
Moreover, the length and the composition factors of $M$ in $\Mod\EE$ coincides with those in $\AA$.\end{lemma}
\begin{proof}
(1) $\Rightarrow$ (2):
This is clear since the lattice of subobjects of $M$ in $\AA$ is a subposet of that in $\Mod\EE$.

(2) $\Rightarrow$ (1):
It suffices to show that every simple object in $\AA$ is also simple in $\Mod\EE$. Let $M$ be a simple object in $\AA$ and $N$ a non-zero submodule of $M$ in $\Mod\EE$. By using the Yoneda lemma, it is easily checked that there is a non-zero finitely generated submodule $N'$ of $N$ in $\Mod\EE$. Since $N'$ is a finitely generated submodule of $M$, we have that $N'$ belongs to $\AA$ by the assumption. Thus $N'=M$ holds by $N' \leq M$ in $\AA$ and $N'\neq 0$.

The remaining assertions are clear from the above proof and the Jordan-H\"older theorem.
\end{proof}

\begin{proof}[Proof of Proposition \ref{onaji}]
(1)
The category $\eff\EE$ satisfies the conditions in Lemma \ref{elementary} by Propositions \ref{effabelian} and \ref{serre}(1).

(2)
Suppose that $\mod\EE$ is abelian. Then it is well-known that the embedding $\mod\EE \to \Mod\EE$ is exact and that $\mod\EE$ is closed under finitely generated submodules, see \cite{aus66}. Thus Lemma \ref{elementary} applies.
\end{proof}

\medskip\noindent
{\bf Acknowledgement.}
The author would like to express his deep gratitude to his supervisor Osamu Iyama for his support and many helpful comments. This work is supported by JSPS KAKENHI Grant Number JP18J21556.


\begin{thebibliography}{199}
  \bibitem[ASS]{ASS}
 I. Assem, D. Simson, A. Skowro\'nski,
 \emph{Elements of the representation theory of associative algebras Vol. 1}, London Mathematical Society Student Texts, 65,
Cambridge University Press, Cambridge, 2006.

 \bibitem[Au1]{aus66}
 M. Auslander, \emph{Coherent functors},
 Proc. Conf. Categorical Algebra (La Jolla, Calif., 1965) 189--231 Springer, New York.

 \bibitem[Au2]{artin2}
 M. Auslander, \emph{Representation theory of Artin algebras II}, Comm. Algebra 1 (1974), 269--310.

 \bibitem[Au3]{relations}
 M. Auslander, \emph{Relations for Grothendieck groups of Artin algebras}, Proc. Amer. Math. Soc. 91 (1984), no. 3, 336--340.

 \bibitem[Au4]{isolated}
 M. Auslander, \emph{Isolated singularities and existence of almost split sequences}, Representation theory, II (Ottawa, Ont., 1984), 194--242, Lecture Notes in Math., 1178, Springer, Berlin, 1986.

 \bibitem[AB]{ab}
 M. Auslander, R-O. Buchweitz, \emph{The homological theory of maximal Cohen-Macaulay approximations}, M\'em. Soc. Math. France (N.S.) No. 38 (1989), 5--37.

 \bibitem[AR1]{ar1}
 M. Auslander, I. Reiten, \emph{Grothendieck groups of algebras and orders},
 J. Pure Appl. Algebra 39 (1986), no. 1--2, 1--51.

 \bibitem[AR2]{applications}
 M. Auslander, I. Reiten, \emph{Applications of contravariantly finite subcategories}, Adv. Math. 86 (1991), no. 1, 111--152.

 \bibitem[But]{but}
 M. C. R. Butler, \emph{Grothendieck groups and almost split sequences}, Lecture Notes in Math., 882, Springer, Berlin-New York, 1981.

 \bibitem[B\"u]{buhler}
 T. B\"uhler, \emph{Exact categories}, Expo. Math. 28 (2010), no. 1, 1--69.

 \bibitem[En]{en}
 H. Enomoto, \emph{Classifying exact categories via Wakamatsu tilting}, J. Algebra 485 (2017), 1--44.

 \bibitem[En2]{en2}
 H. Enomoto, \emph{Classifications of exact structures and Cohen-Macaulay-finite algebras}, arXiv:1705.02163.

 \bibitem[Fi]{fiorot}
 L. Fiorot, \emph{N-Quasi-Abelian Categories vs N-Tilting Torsion Pairs}, arXiv:1602.08253.

 \bibitem[Gr]{gro}
 A. Grothendieck, \emph{Sur quelques points d'alg\`ebre homologique}, T\^ohoku Math. J. (2) 9 1957 119--221.

 \bibitem[Hi]{hir}
 N. Hiramatsu, \emph{Relations for Grothendieck groups of Gorenstein rings}, Proc. Amer. Math. Soc. 145 (2017), no. 2, 559--562.

 \bibitem[Kr]{krause}
 H. Krause, \emph{Krull-Schmidt categories and projective covers}, Expo. Math. 33 (2015), no. 4, 535--549.

 \bibitem[Ko]{ko}
 T. Kobayashi, \emph{Syzygies of Cohen-Macaulay modules and Grothendieck groups}, J. Algebra 490 (2017), 372--379.

 \bibitem[MMP]{mmp}
 E. N. Marcos, H. A. Merklen, M. I. Platzeck, \emph{The Grothendieck group of the category of modules of finite projective dimension over certain weakly triangular algebras}, Comm. Algebra 28 (2000), no. 3, 1387--1404.

 \bibitem[PR]{pr}
 M. I. Platzeck, I. Reiten, \emph{Modules of finite projective dimension for standardly stratified algebras}, Comm. Algebra 29 (2001), no. 3, 973--986.

 \bibitem[Sm]{smalo}
 S. O. Smal\o, \emph{Torsion theories and tilting modules}, Bull. London Math. Soc. 16 (1984), no. 5, 518--522.

 \bibitem[We]{weibel}
 C. Weibel, \emph{An introduction to homological algebra}, Cambridge Studies in Advanced
Mathematics, 38, Cambridge University Press, Cambridge, 1994.

\end{thebibliography}
\end{document}